\documentclass[11pt,leqno]{amsart}%
\usepackage{amsmath}
\usepackage{amsfonts}
\usepackage{amssymb}
\usepackage{graphicx}%
\providecommand{\U}[1]{\protect\rule{.1in}{.1in}}

\newtheorem{theorem}{Theorem}

\newtheorem{corollary}[theorem]{Corollary}

\newtheorem{definition}[theorem]{Definition}

\newtheorem{proposition}[theorem]{Proposition}
\newtheorem{remark}[theorem]{Remark}

\begin{document}
\title[$L^1$-Liouville property and stochastic incompleteness]{On the $L^{1}$-Liouville property of stochastically incomplete manifolds}
\author{G. Pacelli  Bessa}
\address{Departamento de Matem\'atica \\Universidade Federal do Cear\'a-UFC\\
60455-760 Fortaleza, CE, Brazil}
\email{bessa@mat.ufc.br}
\author{Stefano Pigola}
\address{Sezione di Matematica - DiSAT\\
Universit\`a dell'Insubria - Como\\
via Valleggio 11\\
I-22100 Como, ITALY}
\email{stefano.pigola{@}uninsubria.it}
\author{Alberto G. Setti}
\address{Sezione di Matematica - DiSAT\\
Universit\`a dell'Insubria - Como\\
via Valleggio 11\\
I-22100 Como, ITALY}
\email{alberto.setti@uninsubria.it}
\begin{abstract}
A classical result by Alexander Grigor'yan states that on a stochastically complete manifold the non-negative superharmonic  $L^1$-functions are necessarily constant. In this paper we address the question of whether and to what extent the reverse implication holds.
\end{abstract}
\date{\today}
\thanks{The first named author acknowledges the hospitality of the University of Insubria}
\maketitle
\tableofcontents

\section*{Introduction}

Let $\left(M,g\right)  $ be an $m$-dimensional Riemannian manifold. We use
the symbol $\Delta$ to denote the negative-definite Laplace-Beltrami operator of $M$.
Thus, if $M=\mathbb{R}$, $\Delta=+d^{2}/dx^{2}$. By a superharmonic function
we mean a function $u\in C^{0}\left(  M\right)  \cap W_{loc}^{1,2}\left(
M\right)  $ satisfying $\Delta u\leq0$ in the sense of distributions, namely,%
\[
-\int_{M}\left\langle \nabla u,\nabla\varphi\right\rangle dv\leq0,
\]
for every $\varphi\in W_{c}^{1,2}\left(  M\right)  $. By reversing the
inequality we obtain the notion of subharmonic function and by replacing the
inequality with an equality we get a harmonic function. By elliptic
regularity, harmonic functions are necessarily smooth.

In general, there is no obstruction for a manifold to support many
(super)harmonic functions. Indeed, according to a theorem due to R. E. Greene and H. Wu
\cite{GreeneWu}, any $m$-dimensional Riemannian manifold can  be embedded into $\mathbb{R}^{2m+1}$
by harmonic functions.
On the other hand, the presence of superharmonic functions enjoying some special property is
intimately related to the geometry
of the underlying space. Thus, for instance, if the geodesically complete
manifold $\left(  M, g  \right)  $ supports a
non-constant, positive superharmonic function then $M$ is non-parabolic and,
in particular,%
\[
\int^{+\infty}\frac{r}{\mathrm{vol}\left(  B_{r}\left(  o\right)  \right)
}dr<+\infty,
\]
for some origin $o\in M$. Here, $B_{r}\left(  o\right)  $ denotes the geodesic
ball of $M$ centered at $o$ and of radius $r>0$. In this spirit one gives the following

\begin{definition}
A smooth Riemannian manifold $\left(  M,g\right)  $ is said to satisfy the $L^{1}
$-Liouville property, $($shortly, $M$ is $L^1$-Liouville$)$, if every non-negative superharmonic function $0\leq u\in
L^{1}\left(  M\right)  $ must be constant.
\end{definition}

According to a nice result by Alexander Grigor'yan \cite{grigoryan-Uzvestiya}, later extended
to non-linear operators modeled on the $p$-Laplacian (see \cite{holopainen-israel}, \cite{prs-revista})
in order to understand whether or not a manifold is $L^{1}$-Liouville one may
simply consider the behavior of its Green kernel $G\left(  x,y\right)  $.
We recall that this latter is the minimal, positive, fundamental solution of $-\Delta$.

\begin{theorem}
\label{th_green}The Riemannian manifold $\left(  M,g\right)  $ is $L^{1}%
$-Liouville if and only if, for some (hence any) $x\in M$,
\[
\int_{M}G\left(  x,y\right)  dv\left(  y\right)  =+\infty.
\]

\end{theorem}

Note that, in case that $M$ is parabolic, we have $G\equiv+\infty$ and the integrability condition is
trivially satisfied. However, in this case, we already know that positive
superharmonic functions (without any further restriction) must be constant.

In \cite{grigoryan-Uzvestiya}, A. Grigor'yan makes a clever use of the equivalence
established in Theorem \ref{th_green} to obtain a neat geometric conditions implying
the $L^{1}$-Liouville
property. This is achieved by relating the (non-)integrability of the Green
function with a further stochastic property of the manifold, namely, its
stochastic completeness. Recall that $\left(  M,g\right)  $ is stochastically
complete (for the Brownian motion with infinitesimal generator $\Delta$) if
for some (hence every) $x\in M$,%
\[
\int_{M}p_{t}\left(  x,y\right)  dv\left(  y\right)  =1,
\]
where $p_{t}\left(  x,y\right)  $ stands for the heat kernel of $M,$ i.e., the
minimal, positive fundamental solution of the heat operator $\Delta
-\partial/\partial t$. From the probabilistic viewpoint, this
means that the explosion time of the Brownian motion on $M$ is almost surely infinite.
Recall also that $G\left(  x,y\right)  $ and
$p_{t}\left(  x,y\right)  $ are related by%
\[
G\left(  x,y\right)  =\int_{0}^{+\infty}p_{t}\left(  x,y\right)  dt.
\]
Therefore, applying Tonelli's Theorem, from Theorem \ref{th_green} \ we immediately deduce

\begin{corollary}
\label{prop_stoch&L1}A stochastically complete manifold is $L^{1}$-Liouville.
\end{corollary}

In particular, since a geodesically complete manifold $\left(  M,g\right)  $
is stochastically complete provided, for some origin $o\in M$,%
\begin{equation}
\int^{+\infty}\frac{r}{\log\left(  \mathrm{vol}\left(  B_{r}\left(  o\right)
\right)  \right)  }dr=+\infty\label{growth-stochcompl}%
\end{equation}
one may conclude the validity of the next

\begin{corollary}
A geodesically complete Riemannian manifold $\left(  M,g\right)  $ is $L^{1}$-Liouville
provided the volume growth condition (\ref{growth-stochcompl}) is satisfied.
\end{corollary}

So far we have essentially celebrated A. Grigor'yan work on the subject. In this
paper we address the following

\medskip

\noindent \textbf{Problems }
\textit{$($a$)$ Does the converse of Corollary \ref{prop_stoch&L1} hold? $($b$)$ If not, to what
extent and under which conditions the validity of the $L^{1}$-Liouville
property implies that the manifold is stochastically complete?}

\medskip

In the case of a model manifold
\[
M_{\sigma}^{m}=\left(  [0,+\infty)\times\mathbb{S}^{m-1},dt^{2}+\sigma
\left(  t\right)  ^{2}d\theta^{2}\right)  ,
\]
it is easy to see that stochastic completeness
is in fact equivalent to the $L^1$-Liouville property. Indeed,
the Green's kernel with pole at $o$  of  $M_\sigma ^m$
is given by
\[
G(x,o) = c_m \int_{r}^{+\infty}
\frac 1{\sigma^{m-1}(t)  }dt
\]
so that, interchanging the order of integration,
\[
\begin{split}
\int_M G(x,o) dx &= c_m \int_0^\infty \sigma^{m-1} (r)dr
\int_r ^{\infty}\frac 1 {\sigma^{m-1}(t)}dt \\ &= c_m
\int_0^{+\infty} dt \frac{\int_0^t \sigma^{m-1}(r)dr}
{\sigma^{m-1}(t)},
\end{split}
\]
which shows that the condition for stochastic completeness and that
for the validity of the $L^1$-Liouville property of a model manifold coincide.

The investigation around these very natural questions would benefit of
different viewpoints on the notion of stochastic completeness. We will make a
constant use of the following equivalent description in the language of
maximum principles at infinity (see \cite{PRS-PAMS}, \cite{PRS-Memoirs}).

\begin{theorem}
\label{th_stoch&maxprinciple}A Riemannian manifold $\left(
M,g \right)  $ is stochastically complete if and only if, for every
$u\in C^{2}\left(  M\right)  $ satisfying $\sup_{M}u=u^{\ast}<+\infty$, there
exists a sequence $\left\{  x_{k}\right\}  \subset M$ \ along which%
\[
\text{(i) }u\left(  x_{k}\right)  >u^{\ast}-\frac{1}{k}\text{,\qquad(ii)
}\Delta u\left(  x_{k}\right)  <\frac{1}{k}.
\]

\end{theorem}

\section{An example} \label{section_example}

This section is devoted to show that, in general, an $L^{1}$-Liouville
manifold may be stochastically incomplete. This answers in the negative Problem (a) stated in the previous section. To this end, we construct an
explicit example in two steps.

\medskip

\noindent\textbf{First Step.} Recall that the connected sum $M_{1}\#M_{2}$ of
equidimensional Riemannian manifolds is stochastically incomplete provided
either $M_{1}$ or $M_{2}$ are stochastically incomplete. See \cite[Lemma 3.1]{BessaBar}.
This is a very special case of the following general fact which follows quite easily using
the viewpoint of Theorem \ref{th_stoch&maxprinciple}.

\begin{proposition}
\label{prop_stochends}Let $\left(  M,g\right)  $ be a complete manifold and
let $E_{1},...,E_{k}$ be the ends of $M$ with respect to any smooth, compact
domain $\Omega\subset M$. Then $M$ is stochastically complete if and only if,
for every $j=1,...,k$, either of the following conditions is verified:

\begin{enumerate}
\item[(i)] There exists a compact domain $D_{j}$ together with a
diffeomorphism $f_{j}:\partial D_{j}\rightarrow\partial E_{j}$ such that the
gluing $M_{j}=D_{j}\cup_{f_{j}}E_{j}$ is a stochastically complete manifold
(without boundary).

\item[(ii)] The Riemannian double $\mathcal{D}\left(  E_{j}\right)  $ is a
stochastically complete manifold (without boundary).
\end{enumerate}
\end{proposition}

In particular, consider the $2$-dimensional model manifolds%
\[
M_{\sigma_{j}}^{2}=\left(  [0,+\infty)\times\mathbb{S}^{1},dt^{2}+\sigma
_{j}\left(  t\right)  ^{2}d\theta^{2}\right)  ,
\]
$j=1,2,$ where we require%
\[
\int^{+\infty}\sigma_{1}\left(  t\right)  dt=+\infty
\]
and%
\[
\int^{+\infty}\frac{\int_{0}^{r}\sigma_{2}\left(  t\right)  dt}{\sigma
_{2}\left(  r\right)  }dr<+\infty.
\]
The first condition means that $M_{\sigma_{1}}^{2}$ has infinite volume. On
the other hand, by a well known characterization, \cite[Prop. 3.2]{grigoryan-BAMS}, the second condition is
equivalent to requiring that $M_{\sigma_{2}}^{2}$ be stochastically incomplete.
Let%
\[
M=M_{\sigma_{1}}^{2}\#M_{\sigma_{2}}^{2}%
\]
where the connected sum is performed using embedded disks $D$ centered at the
poles of the manifolds. By the above considerations, $\left(  M,g\right)  $ is
stochastically incomplete. In particular, $M$ is non-parabolic, therefore,
it possesses a Green function $G<+\infty$.

\medskip

\noindent\textbf{Second Step. }We now perform a conformal change of the metric
$g$. We define%
\[
\tilde g=\lambda^{2}g,
\]
where $\lambda>0$ is any smooth function with the following properties:

\begin{enumerate}
\item[(a)] Outside a neighborhood of $M_{\sigma_{1}}^{2}\backslash D\subset
M$, $\lambda\equiv1.$
\item[]
\item[(b)] Outside a neighborhood of $M_{\sigma_{2}}^{2}\backslash D\subset
M$, $\lambda$ satisfies%
\[
\lambda\left(  t,\theta\right)  \geq\frac{1}{\sqrt{\min_{[t_{1},t]\times
\mathbb{S}^{1}}G_{x_{0}}\left(  x\right)  }},
\]
where $x_{0}$ is any point in $M_{\sigma_{2}}^{2}\backslash D\subset M$ and,
without loss of generality, $[t_{1},+\infty)\times\mathbb{S}^{1}\subset
M_{\sigma_{1}}^{2}$ has the original metric $dt^{2}+\sigma_{1}^{2}\left(
t\right)  d\theta^{2}$.
\end{enumerate}

\noindent\textbf{Conclusion.} We claim that $\tilde M=\left(  M,\tilde g\right)  $ is stochastically incomplete and possesses the $L^{1}$-Liouville
property. Indeed, according to (a), and using Proposition \ref{prop_stochends}%
, we see that $\tilde M$ is stochastically incomplete. In particular,
$\tilde M$ is non-parabolic. Actually, since%
\[
\Delta_{\tilde g}=\frac{1}{\lambda^{2}}\Delta_{g},
\]
it follows that the Green function $\tilde G$ of $\tilde M$ satisfies%
\[
\tilde G=G.
\]
Therefore,%
\begin{align*}
\int_{\tilde M}\tilde G\left(  x_{0},y\right)  d\tilde v\left(
y\right)   &  =\int_{\tilde M}G\left(  x_{0},y\right)  \lambda^{2}\left(
y\right)  dv\left(  y\right) \\
&  \geq\lim_{t\rightarrow+\infty}\int_{[t_{1},t]\times\mathbb{S}^{1}\subset
M_{\sigma_{1}}^{2}}G\left(  x_{0},y\right)  \lambda^{2}\left(  y\right)
dv\left(  y\right) \\
&  \geq C\int_{t_{1}}^{+\infty}\sigma_{1}\left(  t\right)  dt\\
&  =+\infty,
\end{align*}
and by Theorem \ref{th_green} we conclude that the Riemannian manifold
$\tilde M$ is $L^{1}$-Liouville.

Actually, a variation of the above construction allows us to produce an example of a stochastically incomplete  $L^1$-Liouville manifold with only one end.
As above, we start with a $2$-dimensional stochastically incomplete model
\[
M_{\sigma}^{2}=\left(  [0,+\infty)\times\mathbb{S}^{1},g= dr^{2}+\sigma\left(  r\right)  ^{2}d\theta^{2}\right),
\]
with $\sigma (r)$ increasing and diverging to infinity at infinity, and (radial) Green's function with pole at $o$,  $G(o,x)= G(o,r(x))$, and perform the conformal change
of metric
\[
\tilde g = \lambda g
\]
with a conformality factor $\lambda(x)\geq 1$ such that
$\lambda (x)=1$ if $x=re^{i\theta}$  with  $-\pi/2\leq \theta \leq \pi/2$ and $\lambda (x)\geq G(o, r)^{-1/2}$ if $x=re^{i\theta}$ with
$r>1$ and $3\pi/4\leq \theta \leq 5\pi/4$. Denoting as above with a tilde the quantities relative to the conformal metric $\tilde g$ and, using again the fact that $\tilde G(o,x)= G(o,x)$  and that
$d\tilde v = \lambda dv= \lambda\, \sigma dr d\theta$ we see that
\begin{equation*}
\begin{split}
\int_{\tilde M} \tilde G(o, x) d\tilde v &\geq \int_{[1,\infty)\times [3\pi/4,5\pi/4]} G\left(o,re^{i\theta}\right) \lambda\left(re^{i\theta}\right) \sigma (r)drd\theta \\
& \geq  \pi/2 \int_1^\infty \sigma(r) dr = +\infty,
\end{split}
\end{equation*}
and $\tilde M$ is $L^1$-Liouville. On the other hand, let
\[
v_o(r)=\int_0^r \sigma(t)^{-1} \int_0^t \sigma (s)ds\, dt,
\]
and let $v\left(re^{i\theta}\right)=v_o(r)\cos(\theta)$. Then $v$ tends to its supremum along the ray $re^{i0}$ and, using $\tilde \Delta = \lambda ^{-2}\Delta$ and $\Delta v_o=1$
we deduce that in the region where $-\pi/4\leq \theta\leq \pi/4$ and $\sigma(r)^2>2 \sup v_o$ we have
\begin{equation*}
\tilde \Delta v\left(re^{i\theta}\right)  =
\frac 1 {\lambda (re^{i\theta})^{2}}
\left( \Delta v_o(r) \cos(\theta) -  \frac 1 {\sigma(r)^{2}} v_o(r)\cos (\theta)\right) \geq  \frac{\sqrt 2} 4.
\end{equation*}
Thus $v$ does not satisfy the maximum principle at infinity and $\tilde M$ is not stochastically complete.

\par

The examples above stress the fact that equivalence between stochastic completeness and the validity of the
$L^1$-Liouville property depends very much on the rotational invariance of the models. In the presence of a strong anisotropy, it is  possible that Brownian motion may explode in finite time in certain directions and yet
the $L^1$-Liouville property holds, due to the fact that the Green's kernel is big enough in other regions (or ends of the manifold).

The first example constructed in Section \ref{section_example} fits very well in this order of ideas. In fact, inspection of that example shows that the stochastically incomplete end remains essentially untouched whereas the background metric is conformally modified only on the end responsible for the validity of the $L^{1}$-Liouville property.

While, the second example shows that the $L^1$-Liouville property does not imply even a weak form of
stochastic completeness where it is required that at least one of the ends of the manifold is stochastically
complete.

\section{Mean exit time and the $L^{1}$-Liouville property} \label{section_meanexit}

As remarked above, stochastic completeness and the $L^1$-Liouville property
are equivalent on models, but, in general, $L^{1}$-Liouville manifolds may be  stochastically
incomplete. We are thus naturally led to investigate general geometric conditions that
guarantee that a a stochastically incomplete manifold is not
$L^1$-Liouville. In this section we will focus our attention
on curvature conditions both of intrinsic and of extrinsic nature. In both
cases we shall use the notion of \textquotedblleft global mean exit
time\textquotedblright\ that we are going to introduce.\medskip

Let $\left(  M,g\right)  $ be a complete Riemannian manifold and let $o\in M$
be a fixed reference point. The mean exit time of the Brownian motion from the
ball $B_{R}\left(  o\right)  $ \ is defined as the (positive) solution of the
Dirichlet problem%
\[
\left\{
\begin{array}
[c]{ll}%
\Delta E_{R}=-1 & \text{on }B_{R}\left(  o\right) \\
E_{R}=0 & \text{on }\partial B_{R}\left(  o\right)  .
\end{array}
\right.
\]
Note that $E_{R}$ is a smooth function on $B_{R}\left(  o\right)  $. Moreover,
if $G_{R}\left(  x,y\right)  $ denotes the Dirichlet Green function of
$B_{R}\left(  o\right)  $, then, we have the representation formula%
\[
E_{R}\left(  x\right)  =\int_{B_{R}\left(  o\right)  }G_{R}\left(  x,y\right)
dv\left(  y\right)  .
\]
Since $G_{R}\left(  x,y\right)  \nearrow G\left(  x,y\right)  $ as
$R\nearrow+\infty$, by monotone convergence we deduce that%
\[
E_{R}\left(  x\right)  \nearrow E\left(  x\right)  =\int_{M}G\left(
x,y\right)  dv\left(  y\right)  .
\]
We call $E\left(  x\right)  $ the \textit{global mean exit time of $M$}. With this
terminology and notation, $M$ is not $L^{1}$-Liouville if and only if $E$ is a
genuine (say, finite) function. In particular, on a stochastically complete manifold,
the global mean exit time must be infinite. On the other hand, we point out that, according to Section \ref{section_example}, there exist stochastically incomplete manifolds with infinite global mean exit time, thus showing that, in general, the global mean exit time does not carry enough information on the explosion of the Brownian motion because of the possible presence of direction along which explosion can occur in finite time.

\subsection{Intrinsic curvature restrictions}

We shall prove the following

\begin{theorem}
\label{th_curv&l1}Let $\left(  M,g\right)  $ be a complete Riemannian manifold
of dimension $m$ with a pole $o$. Assume that the distance
function $r\left(  x\right)  =d_{M}\left(  x,o\right)  $ satisfies%
\begin{equation}
\Delta r\geq\left(  m-1\right)  \frac{\sigma^{\prime}}{\sigma}\text{, on
}M\label{comp-lap}%
\end{equation}
where $\sigma:[0,+\infty)\rightarrow\lbrack0,+\infty)$ is the warping function
of the $m$-dimensional model manifold $M_{\sigma}^{m}.$ If $M_{\sigma}^{m}$ is
not $L^1$-Liouville (equivalently, stochastically incomplete) then $M$ is not
$L^{1}$-Liouville.
\end{theorem}

\begin{remark}
\rm{By standard comparison arguments, condition (\ref{comp-lap}) follows from the
radial sectional curvature condition%
\[
Sec_{rad}\left(  x\right)  \leq-\frac{\sigma^{\prime\prime}}{\sigma}\left(
r\left(  x\right)  \right)  .
\]
Note that the above result allows us to recover the already noted
equivalence of stochastic completeness and $L^1$-Liouville property
of model manifolds.}
\end{remark}

\begin{proof}
Define%
\begin{equation}
\label{F_R}
F_{R}\left(  r\right)  =\int_{r}^{R}\frac{\int_{0}^{t}\sigma^{m-1}\left(
s\right)  ds}{\sigma^{m-1}\left(  t\right)  }dt
\end{equation}
and%
\begin{equation}
\label{F}
F\left(  r\right)  =\int_{r}^{+\infty}\frac{\int_{0}^{t}\sigma^{m-1}\left(
s\right)  ds}{\sigma^{m-1}\left(  t\right)  }dt.
\end{equation}
Since, by assumption, $M_{\sigma}^{m}$ is stochastically incomplete, we have%
\[
F\left(  r\right)  <+\infty,\text{ }\forall r.
\]
Direct computations  show that the transplanted function $F_{R}\left(
r\left(  x\right)  \right)  $ satisfies%
\[
\left\{
\begin{array}
[c]{llll}%
\Delta F_{R}&\leq & -1 & \text{on }B_{R}\left(  o\right)  \\
F_{R}&=&0 & \text{on }\partial B_{R}\left(  o\right)  .
\end{array}
\right.
\]
Therefore, by comparison on bounded domains,%
\[
E_{R}\left(  x\right)  \leq F_{R}\left(  r\left(  x\right)  \right)  \text{ on
}B_{R},
\]
and letting $R\rightarrow+\infty$ we conclude%
\[
E\left(  x\right)  \leq F\left(  r\left(  x\right)  \right)  .
\]
This proves that $M$ is not $L^{1}$-Liouville.
\end{proof}

\subsection{Minimal submanifolds}

This subsection aims to showing that
the $L^1$-Liouville property of a proper minimal submanifold $\Sigma$
of manifold with a pole $N$ depends on the curvature of the
ambient space. In particular, in the case where  $N$ is a model with warping function
$\sigma$, if the $m$-dimensional model manifold $M^m_\sigma$ is not
$L^1$-Liouville, and $\sigma$ satisfies a technical condition, then
$\Sigma$ is not $L^1$-Liouville.
As alluded to above, we shall use a global mean exit time comparison
argument which extends a previous result by S. Markovsen, \cite{markvorsen}
(see also \cite{BeMo-BLMS}).

\begin{theorem} \label{th_minimal}
Let $f:\Sigma\hookrightarrow N$ be an $m$-dimensional properly immersed minimal
submanifold into an $n$-dimensional Riemannian  manifold $N$. Assume that
the sectional curvature of $N$
satisfies
\[
K^N\leq -G(\rho(y))
\]
where $G$ is a smooth  even function on $\mathbb{R}$ and
$\rho\left(  y\right)  =d_N\left(  y,p\right)$ denotes the
Riemannian distance function from a fixed point $p\in N$. Let $\sigma$
be the solution of the initial value problem
\[
\begin{cases}
\sigma''= G\sigma &\\
\sigma(0)=0, \, \sigma'(0)=1&
\end{cases}
\]
and assume that
\[
\sigma^{\prime}\geq0
\]
and%
\begin{equation}
\label{sigma_condition}
\frac{\sigma^{m-1}}{\int_{0}^{t}\sigma^{m-1}}\leq m\,\frac{\sigma^{\prime}}{\sigma},
\end{equation}
on $(0,+\infty)$. Then, for every extrinsic ball $B_{R}^{N}$ centered at  $p$ in
$N^n$ with radius $R<\mathrm{inj}_{N}(p)$, the mean exit time
$E_{f^{-1}(B_{R}^N)}(x)$ satisfies
\[
E_{f^{-1}(B_{R}^{N})}\left(  x\right)  \leq F_{R}\circ\rho\circ f\left(
x\right)  ,
\]
where $F_R$ is defined in \eqref{F_R}.
In particular, if $\mathrm{inj}_N(p)=+\infty$, and
\begin{equation}
\label{F_integrability}
t\to\frac{\int_{0}^{t}\sigma^{m-1}}{\sigma^{m-1}}\in L^1(+\infty)
\end{equation}
then $\Sigma$ is not $L^{1}$-Liouville.
\end{theorem}

\begin{proof}
S. Markvorsen \cite{markvorsen}, considered the case of constant curvature reference spaces, which correspond
to the choices $$\sigma(t)=t,\,\,\, \sigma(t)= k^{-1}\sin(kt), \,\,\, \sigma(t)={k}^{-1}\sinh(kt),$$ with $k>0$.
Actually, it is possible to
extend Markvorsen arguments to the more general setting of Theorem
\ref{th_minimal} by using the function
\[
\overline{F}_{R}\left(  t\right)  =F_{R}\circ \mathrm{i}^{-1}\left(  t\right)  ,
\]
with%
\[
\mathrm{i}\left(  t\right)  =\int_{0}^{t}\sigma\left(  s\right)  ds.
\]
We are going to exhibit a more straightforward argument which avoids
the use of the auxiliary function $\overline{F}_{R}$.

Recall that if $f:M\to N$ is an isometric immersion  and $\varphi:N\to
\mathbb{R}$ and $F:\mathbb{R}\to \mathbb{R}$ are smooth, then for
every $X\in T_xM$ we have
\[
\begin{split}
\mathrm{Hess}(F\circ\varphi\circ f) (X,X) &=
F''(\varphi(f(x))) \langle \nabla^N \varphi, df X\rangle^2\\
&+F'(\varphi(f(x))) \bigl[
\mathrm{Hess}^N \varphi (df X,df X) +\langle \nabla^N \varphi , II
(X,X)\rangle
\bigr].
\end{split}
\]
If  $\varphi = \rho$ is the distance function, then the assumption
on the sectional curvature of $N$ implies
\[
\mathrm{Hess}^N \rho (Y,Y) \geq \frac{\sigma'}\sigma \bigl[\langle Y, Y\rangle -
\langle \nabla^N \rho, Y\rangle^2\bigr].
\]
Thus, assuming that $f$ is minimal and that $F'\leq 0$ and setting for ease of notation
$\rho_x=\rho(f(x))$,  we obtain
\[
\Delta (F\circ \rho\circ f)(x) \leq m \bigl (F'\frac {\sigma'}\sigma\bigr)(\rho_x)+
\bigl(F''-F'\frac{\sigma'}\sigma \bigr)(\rho_x) \sum_{i=1}^m \langle
\nabla^N \rho, df X_i\rangle^2,
\]
where $\{X_i\}$ is an orthonormal basis on $T_xM$.

Now, if $F=F_R$, then we have
\[
F_R' (r) = -
\frac{\int_{0}^{t}\sigma^{m-1}\left(  s\right)
ds}{\sigma^{m-1}\left(  t\right)  }<0
\]
and
\[
F''_R(r) = -1 - (m-1) \frac {\sigma'}\sigma F'_R,
\]
so substituting,
\[
\Delta (F\circ \rho\circ f)(x) \leq m \bigl (F'\frac
{\sigma'}\sigma\bigr)(\rho_x)- \bigl[(1+ m\frac{\sigma'}\sigma
F_R'(\rho_x)\bigr] \sum_{i=1}^m \langle
\nabla^N \rho, df X_i\rangle^2.
\]
Complete $\{df X_i\}_{i=1}^m$ to an orthonormal basis
$\{df X_i\}_{i=1}^m\cup \{X_j\}_{j=m+1}^n$ on $T_{f(x)}N$,  and
note that
\[
\sum_{i} \langle \nabla^N\rho, df X_i\rangle^2 + \sum_j \langle \nabla^N\rho,
Y_j\rangle^2 =1.
\]
Inserting this in the above inequality and using the assumption
\[
m\frac{\sigma'}\sigma F'_R=-m\frac{\sigma'}\sigma\frac{\int_{0}^{t}\sigma^{m-1}\left(  s\right)
ds}{\sigma^{m-1}\left(  t\right)  } \leq -1
\]
we finally obtain
\[
\begin{split}
\Delta (F\circ \rho\circ f)(x) &\leq
-\sum_i \langle \nabla^N\rho , df X_i\rangle^2 + m\frac{\sigma'}\sigma F'_R
\sum_j\langle \nabla^N\rho,Y_j\rangle^2 \\
&\leq
-\sum_i \langle \nabla^N\rho , df X_i\rangle^2  -
\sum_j\langle \nabla^N\rho,Y_j\rangle^2 = -1.
\end{split}
\]
Thus,
\[
\Delta\left(  F_{R}\circ\rho\circ f\right)  \leq-1\text{ on }f^{-1}%
(B_{R}^{N_{g}}),
\]
and since both $E_{f^{-1}(B_{R}^{N})}$ and $F_{R}\circ\rho\circ f$
vanish on  $\partial f^{-1}(B_{R}^N)$, the first assertion in the statement
follows from the comparison principle.

The second assertion follows  letting $R\to +\infty$, so that $f^{-1}(B_{R}^N)\nearrow M$
and $E_R (x)\nearrow E(x)$, while
\[
F_R(r) \nearrow \int_{r}^{R}\frac{\int_{0}^{t}\sigma^{m-1}\left(  s\right)
ds}{\sigma^{m-1}\left(  t\right)  }dt,
\]
and, as seen above, the assumption that $M^m_\sigma$ is not
$L^1$-Liouville amounts to the fact that the integral on the right
hand side is finite.
\end{proof}

Note that \eqref{F_integrability} amounts to the fact that the
$m$-dimensional model manifold $M_\sigma^m$ is not $L^1$-Liouville.

We also remark that, since $\sigma$ is non-decreasing,  for
every $n\geq m$
\[
\int_{r}^{R}\frac{\int_{0}^{t}\sigma^{n-1}\left(  s\right)
ds}{\sigma^{n-1}\left(  t\right)  }dt
\leq
\int_{r}^{R}\frac{\int_{0}^{t}\sigma^{m-1}\left(  s\right)
ds}{\sigma^{m-1}\left(  t\right)  }dt,
\]
so condition \eqref{F_integrability} also implies that the n-dimensional model
$N^n_\sigma$ is not $L^1$-Liouville, and  therefore, by Theorem~\ref{th_curv&l1}
the same holds for the manifold $N$.

\end{document}